\documentclass[11pt]{amsart}
\usepackage[utf8]{inputenc}
\usepackage[margin=3cm]{geometry}
\usepackage{amsmath,amssymb,amsthm}
\usepackage{enumerate}
\usepackage{graphicx}
\usepackage[english]{babel}
\usepackage[pdfborder={0 0 0}]{hyperref}
\usepackage[utf8]{inputenc}
\usepackage{verbatim}
\usepackage{color}
\usepackage{chngcntr}
\usepackage{pgfplots}
\pgfplotsset{compat=1.15}
\usepackage{mathrsfs}
\usetikzlibrary{arrows}

\usepackage{cleveref}
\usepackage{hyperref}

\newtheorem{theorem}{Theorem}[section]
\newtheorem*{theorem*}{Theorem}
\newtheorem{proposition}[theorem]{Proposition}

\newtheorem{lemma}[theorem]{Lemma}
\newtheorem{conjecture}[theorem]{Conjecture}
\newtheorem*{conjecture*}{Conjecture}
\newtheorem{fact}[theorem]{Fact}
\newtheorem{question}[theorem]{Question}
\newtheorem*{question*}{Question}
\newtheorem{claim}[theorem]{Claim}
\theoremstyle{definition}

\newcommand{\F}{\mathcal{F}}
\newcommand{\A}{\mathcal{A}}
\newcommand{\B}{\mathcal{B}}

\newcommand{\Z}{\mathbb{Z}}

\DeclareMathOperator{\La}{La}
\DeclareMathOperator{\mad}{D}
\newcommand{\Ssize}{s}
\newcommand{\comporder}{t}
\newcommand{\compsize}{m}

\title{Largest component in Boolean sublattices}
\thanks{This work was carried out in the context of the master's thesis research of JG, see~\protect\url{https://scripties.uba.uva.nl/search?id=record_54316}. RJK was partially supported by grant OCENW.M20.009 and the Gravitation Programme NETWORKS (024.002.003) of the Dutch Research Council (NWO)}
\author{
Julian Galliano}
\author{
Ross J. Kang
}
\address{Korteweg--de Vries Institute for Mathematics, University of Amsterdam, Netherlands. } 
\email{
\protect\href{mailto:julian.galliano@student.uva.nl }{\protect\nolinkurl{julian.galliano@student.uva.nl}},\protect\href{mailto:r.kang@uva.nl}{\protect\nolinkurl{r.kang@uva.nl}}
}

\begin{document}

\begin{abstract}
For a subfamily $\F\subseteq 2^{[n]}$ of the Boolean lattice, consider the graph $G_\F$ on $\F$ based on the pairwise inclusion relations among its members. Given a positive integer $\comporder$, how large can $\F$ be before $G_\F$ must contain some component of order greater than $\comporder$?
For $\comporder=1$, this question was answered exactly almost a century ago by Sperner: the size of a middle layer of the Boolean lattice. For $\comporder=2^n$, this question is trivial. We are interested in what happens between these two extremes.
For $\comporder=2^{g}$ with $g=g(n)$ being  any integer function that satisfies $g(n)=o(n/\log n)$ as $n\to\infty$, we give an asymptotically sharp answer to the above question: not much larger than the size of a middle layer.
This constitutes a nontrivial generalisation of Sperner's theorem.
We do so by a reduction to a Tur\'an-type problem for rainbow cycles in properly edge-coloured graphs.
Among other results, we also give a sharp answer to the question, how large can $\F$ be before $G_\F$ must be connected?
\end{abstract}


\date{17 March 2025; this is the Author Accepted Manuscript (AAM) for {\em Acta Mathematica Hungarica}}
\maketitle


\section{Introduction}

Sperner's theorem~\cite{Spe28} is a revered result in combinatorics that
gives a precise answer to the following problem.
Find the largest number $\La(n)$ such that there exists some family $\F\subseteq 2^{[n]}$ of subsets of the $n$-element ground set $[n]=\{1,\dots,n\}$ no two of which are comparable, i.e.~there are no $A,B\in\F$ for which either $A\subseteq B$ or $B\subseteq A$, satisfying that $|\F|=\La(n)$. 
Equivalently, find the smallest number $\La(n)$ such that in any family $\F\subseteq 2^{[n]}$ of subsets of $[n]$ with at least $\La(n)+1$ members there must be a pair $A,B\in \F$ of comparable members, so $A\subseteq B$ or $B\subseteq A$, otherwise known as a {\em $2$-chain}.
In this context, $2^{[n]}$ is often referred to as the {\em Boolean lattice}.

For our purposes, we prefer to cast $\La(n)$ as a {\em threshold}. There is some threshold value $\La(n)$ such that in any family $\F\subseteq 2^{[n]}$ of subsets of $[n]$ with at least $\La(n)+1$ members there must be a $2$-chain, while, on the other hand, there exists some family $\F\subseteq 2^{[n]}$ of subsets of $[n]$ with only $\La(n)$ members that contains no $2$-chain.
Sperner's theorem~\cite{Spe28} asserts that this threshold satisfies $\La(n)=\binom{n}{\lfloor n/2 \rfloor}=\binom{n}{\lceil n/2 \rceil}$. The extremal family is $\binom{[n]}{\lfloor n/2 \rfloor}$, all subsets of $[n]$ of size $\lfloor n/2 \rfloor$, or, if $n$ is odd, the symmetric one $\binom{[n]}{\lceil n/2 \rceil}$, the so-called {\em middle layer(s)}.

It has long been known that something remarkable happens at this threshold called {\em supersaturation}. Confirming a conjecture of Erd\H{o}s and Katona, Kleitman~\cite{Kle68} showed that in any family $\F\subseteq 2^{[n]}$ of subsets of $[n]$ with at least $\La(n)+q$ members there must be at least $q(\lfloor\frac{n}{2}\rfloor+1)$ $2$-chains. The extremal family here takes subsets of $[n]$ of sizes as close as possible to $n/2$.
(See~\cite{DGKS14,DGS15,BaWa18,Sam19} for the recent generalisation of this phenomenon to $\ell$-chains, $\ell\ge 3$.)

In other words, as we cross the threshold, not only must there be one $2$-chain, but also there must be {\em many}.
A natural question to ask is, what does the space of $2$-chains look like? Since by supersaturation the volume of this space must be large, must it exhibit some clustering? When we demand the family to be of size at least $2^n$, the space of $2$-chains is trivially connected, so more fully we ask, what happens between the threshold $\La(n) \sim 2^{n+1/2}/\sqrt{\pi n}$ and the extreme $2^n$?
Does a large cluster in the space appear suddenly? Or more gradually?

These musings preoccupy us here.
Since these questions are subject to interpretation, we need to ground our definitions.
Given a family $\F\subseteq 2^{[n]}$, we define an auxiliary graph $G_{\F}$ with vertex set $\F$ and two vertices $A,B\in \F$ are adjacent if $A,B$ comprise a $2$-chain. Of course, if $\F = 2^{[n]}$, then $G_\F$ is a complete graph on $2^n$ vertices.
We focus on the following novel prospect.

\begin{conjecture}
\label{main conjecture}
Let $k,n$ be integers with $0\le k\le n$ such that $k=0$, $k=1$ or $k$ has the same parity as $n$. If for some $\F\subseteq 2^{[n]}$ the components of $G_\F$ have order at most $2^k$, then
    \[
    |\F|\le 2^k\binom{n-k}{\lfloor (n-k)/2\rfloor}.
    \]
\end{conjecture}

\noindent
Note that when $k=0$, then the statement is Sperner's theorem.
The $k=1$ case is corollary of a result of Katona and Tarj\'an~\cite{KaTa83}.
The $k=n$ case corresponds to the complete case $\F=2^{[n]}$, so trivially holds.
A fuller confirmation of the cases $1<k<n$ ---but especially those $k=k(n)$ for which $k=o(n)$ as $n\to\infty$--- would constitute a natural generalisation/extension of Sperner's theorem.

(By Stirling's approximation, note that when $n-k\to\infty$, then for all even $n-k$ we have
\[
2^k\binom{n-k}{\lfloor (n-k)/2\rfloor} \sim 2^k \frac{2^{n-k}}{\sqrt{(n-k)\pi/2}} = \frac{2^n}{\sqrt{(n-k)\pi/2}}
\]
and that this is asymptotically equal to $\La(n)$ when $k=o(n)$ as $n\to\infty$.)

This conjecture seems difficult, but, if it were true, it would be best possible, essentially by a disjoint union of complete subfamilies.
\begin{proposition}\label{prop:sharp}
    For all integers $k,n$ with $0\le k\le n$, there exists a family $\F\subseteq 2^{[n]}$ of size $2^k\binom{n-k}{\lfloor (n-k)/2\rfloor}$ such that the components of $G_\F$ have order $2^k$.
\end{proposition}

\begin{proof}
    Consider the family described as follows. Let $\F_s=\binom{[n-k]}{\lfloor(n-k)/2\rfloor}$, i.e.~the set of all $\lfloor(n-k)/2\rfloor$-element subsets of $[n-k]$. Then let $\F$ be the family of all sets $F$ such that $F_s\subseteq F\subseteq F_s\cup([n]\backslash[n-k])$ for some $F_s\in\F_s$. 
    Then $G_\F$ is made up of exactly $\binom{n-k}{\lfloor (n-k)/2\rfloor}$ components, each of which is identified by a unique collection of $\lfloor (n-k)/2\rfloor$ elements from the set $[n-k]$. Also note that each component has order $2^k$, since each element of a component is determined by a unique subset of $[n]\backslash[n-k]$, of which there are $2^k$. It only remains to show that any two components are disconnected. For this, assume we have $X=F_X\cup X'$ and $Y=F_Y\cup Y'$, with $F_X,F_Y\in\F_s$, and suppose $X\subseteq Y$. Then
    \(
    F_X\subseteq X\subseteq Y\subseteq F_Y\cup([n]\backslash[n-k]).
    \)
    But this is a contradiction if $F_X\ne F_Y$, since then $F_X$ contains some element not contained in $F_Y$, and since $F_X\subseteq [n-k]$, this element cannot be contained in $F_Y\cup ([n]\backslash[n-k])$ either.
\end{proof}
Note these constructions are quite distinct from the extremal families for supersaturation.

For the discussion that follows, let us denote by $\La(n,\comporder)$ the size of the largest family $\F\subseteq 2^{[n]}$ such that all components of $G_\F$ have order at most $\comporder$. So $\La(n,1)=\La(n)$, and \Cref{main conjecture} and \Cref{prop:sharp} posit that $\La(n,2^k) = 2^k\binom{n-k}{\lfloor (n-k)/2\rfloor}$ if $k=0$, $k=1$, or $0\le k\le n$ has the same parity as $n$.
By definition, if a family $\F\subseteq 2^{[n]}$ has at least $\La(n,\comporder)+1$ elements, then $G_\F$ must have a component of order at least $\comporder+1$.
Therefore, understanding the behaviour of $\La(n,\comporder)$ does indeed give us some idea, at least in one interpretation, of how the space of $2$-chains looks as we tune the prescribed number of elements between $\La(n)+1$ and $2^n$, just as we asked above.

In our main result, we establish for a large range of choices of $\comporder$ the asymptotic behaviour of $\La(n,\comporder)$, which aligns with \Cref{main conjecture} and \Cref{prop:sharp}.

\begin{theorem}
    \label{main}
    Let $\comporder=\comporder(n)=O(2^{g(n)})$ for some $g(n)$ that satisfies $g(n)=o(n/\log n)$ as $n\to\infty$. Then $\La(n,\comporder) = (1+o(1)) \La(n)$.
\end{theorem}
\noindent
What \Cref{main} says is that even if we allow for the components in $G_\F$ to be quite large as a function of $n$, the family $\F$ may not grow much larger than a middle layer family.
So our theorem is a qualitative strengthening of Sperner's theorem.

The main idea in the proof is to relate the size of the family to the average number of times a permutation $\sigma\in S_n$ meets elements of the family. In our proof, we draw an unexpected connection to a Tur\'an-type problem, namely, an upper bound on the edge-density of a graph on two consecutive layers of the Boolean lattice. To obtain this upper bound, we invoke a recent result of Alon, Buci\'c, Sauermann, Zakharov, Zamir~\cite{ABSZZ+} concerning rainbow cycles in proper edge-colourings.

We also pursue further partial, but exact results in support of \Cref{main conjecture}. In particular, we are interested in problems near the boundary cases $k=0$ and $k=n$. 
We have verified that the conjecture holds if $k=1$ or $k=2$ (and $n$ is even).
In doing so, we have developed a slight generalisation of the BLYM inequality (\Cref{diamond BLYM} below) which implies \Cref{main conjecture} conditional upon the hypothetical shape of the components in $G_\F$ of an optimal family $\F$. 
(We see this same shape in \Cref{prop:sharp}.)
This generalisation of the BLYM inequality may be of independent interest.
Since the $k=n$ case corresponds to $G_\F$ being trivially connected, we found it interesting to pursue the largest size of $\F$ before $G_\F$ must be connected.

\begin{theorem}
\label{disconnected}
If $\F\subseteq 2^{[n]}$ is such that $G_\F$ is disconnected, then
\[
|\F| \le
\begin{cases}
2^n-2^{n/2+1}+2 & \text{ if $n$ is even, and }\\
2^n-3\cdot 2^{(n-1)/2}+2 & \text{ if $n$ is odd}
\end{cases}
\]
Moreover, these bounds on $|\F|$ are sharp.
\end{theorem}

\subsection*{Structure of the paper}
In \Cref{sec:proof}, we give the proof of \Cref{main}.
In \Cref{sec:exact}, we give and discuss our generalisation of the BLYM inequality and demonstrate its use in the $k=2$ case of \Cref{main conjecture}.
We prove \Cref{disconnected} in \Cref{sec:connected}.

Before continuing, we introduce some notation and preliminary results that will be useful.

\subsection{Notation and preliminaries}\label{sub:prelim}

The {\em Boolean lattice}, often denoted merely by $2^{[n]}$, is the partially ordered set $(2^{[n]},\subseteq)$ made up of all subsets of $[n]=\{1\dots,n\}$, together with the inclusion relation. We are considering subsets $\F$ of the Boolean lattice. When talking about the Boolean lattice, we often use the term {\em layer}, which is all subsets of $[n]$ of a given size, say, $k$, denoted by $\binom{[n]}k$. We imagine the layers to be arranged bottom to top, with the empty set $\emptyset$ at the bottom and the full set $[n]$ at the top. 
We refer to the {\em height} of a given family $\F$, which is the span of the highest and lowest layers that contain sets in $\F$.

Two distinct sets $X,Y\in\F$ are {\em comparable} if either $X\subseteq Y$ or $Y\subseteq X$, otherwise they are {\em incomparable}. We call a tuple $X_1,\dots,X_k\in\F$ of distinct sets such that $X_1\subseteq\dots\subseteq X_k$ a {\em $k$-chain}, and if $k=n$ we call it a {\em maximal chain}. We denote by $G_\F$ the undirected graph which has as vertices the elements of $\F$ and an edge $XY$ for every comparable pair $X\subseteq Y$. We also work with a subgraph $G'_\F$ of $G_\F$ where we only take those edges $XY$ with $|Y|=|X|+1$. This paper concerns the connectivity and component structure of $G_\F$. We are sometimes ambiguous by the use of phrases like, ``the components of $\F$'' or ``$\F$ does not contain some substructure'', which should always be read as statements about $G_\F$.

We denote by $S_n$ the set of permutations of $[n]$. As it turns out, studying how the permutations in $S_n$ interact with the sets in the family $\F$ is useful for results about $G_\F$. Note that we can interchangeably use either permutations or maximal chains, since a permutation $\sigma$ corresponds to the maximal chain $(\emptyset,\{\sigma(1)\},\{\sigma(1),\sigma(2)\},\dots,[n])$. 
We will say that a permutation $\sigma$ {\em meets} a family $\F$ if the intersection of $\F$ and the maximal chain corresponding to $\sigma$ is nonempty.

Throughout, $\log n$ denotes the base-2 logarithm, while $\ln n$ denotes the natural logarithm.

In \Cref{sub:reduction,sub:xistar}, we discuss and make use of a classic result of Kruskal~\cite{Kru63} and Katona~\cite{Kat68}.

\begin{theorem}[\cite{Kru63,Kat68}]\label{thm:KrKa}
    Let $\F\subseteq \binom{[n]}k$ for some $k$. Suppose we write 
    \[
    |\F|=\binom{n_k}k+\binom{n_{k-1}}{k-1}+\dots+\binom{n_j}j
    \]
    for some $n_k\ge n_{k-1}\ge\dots\ge n_j\ge j$. There is a unique way to do this. Then call ${\mathcal S}_r$ the family of sets of size $k-r$ that are subsets of some element of $\F$. Then
    \[
    |{\mathcal S}_r|\ge\binom{n_k}{k-r}+\binom{n_{k-1}}{k-1-r}+\dots+\binom{n_j}{j-r}.
    \]
\end{theorem}

We also refer occasionally to another classic result, which generalises Sperner's theorem, called the BLYM inequality and due, independently, to Yamamoto~\cite{Yam54}, Me\v{s}alkin~\cite{Mes63}, Bollob\'as~\cite{Bol65}, Lubell~\cite{Lub66}.

\begin{theorem}[\cite{Lub66,Bol65,Yam54,Mes63}]\label{thm:BLYM} 
If $\F\subseteq 2^{[n]}$ is a family with no $2$-chain, then, writing $a_k=|\F \cap \binom{[n]}{k}|$,
\[
\sum_{k=0}^n\frac{a_k}{\binom{n}{k}} \le 1.
\]
\end{theorem}

\section{Proof of \Cref{main}}\label{sec:proof}

The section should be read as follows: in \Cref{sub:lubell} we relate the main quantity at the heart of \Cref{main} to the quantity $\Lambda^\star(n,\comporder)$, which is the average number of times a permutation $\sigma\in S_n$ meets the family $\F$. Our goal is then to upper bound $\Lambda^\star(n,\comporder)$.
\Cref{sub:reduction} contains most of the work toward proving \Cref{main}, first in the case where $\comporder$ is constant, then in \Cref{main lemma}, showing how \Cref{main} is true conditional upon a certain bound on a quantity $\xi^\star_n$, which is the maximum edge density of a graph on two consecutive layers of the Boolean lattice. In \Cref{sub:xistar} we discuss various results about bounding $\xi^\star_n(n)$ and wrap up the proof.

\subsection{The Lubell function}\label{sub:lubell}

In the proof of \Cref{main}, we relate the problem to another quantity.
Let $\F\subseteq 2^{[n]}$. For $\sigma\in S_n$, call $T(\sigma)$ the number of times that $\sigma$ meets $\F$, that is $T(\sigma) = |\{F\in\F:\{\sigma(1),\sigma(2),\dots,\sigma(|F|)\}=F\}|$. Then the {\em Lubell function} of $\F$ is
\[
\Lambda(n,\F)=\frac{\sum_{\sigma\in S_n}T(\sigma)}{n!}.
\]
That is, $\Lambda(n,\F)$ measures the average number of times a permutation meets $\F$, over all permutations $\sigma\in S_n$.
We also define 
$\Lambda^\star(n,\comporder)$ to be the maximum of $\Lambda(n,\F)$ over all families $\F$ such that $G_\F$ has components of order at most $\comporder$. 
 The following further justifies the utility of the Lubell function.

\begin{lemma}
\label{Lubell function}
    For $\F\in 2^{[n]}$,
    \[
    \Lambda(n,\F)=\sum_{F\in\F}\frac{1}{\binom{n}{|F|}}.
    \]
\end{lemma}

\begin{proof}
    For a set $F\in\F$, call $T(F)$ the number of permutations $\sigma\in S_n$ that have $F$ as an initial segment. The key observation is that
    \[
    \sum_{\sigma\in S_n}T(\sigma)=\sum_{F\in\F}T(F).
    \]
    This holds because both sums are counting exactly the number of pairs $(\sigma,F)$ for which $F$ is an initial segment of $\sigma$. But, for a given $F$, we can compute $T(F)$ exactly, namely $T(F)=|F|!(n-|F|)!$. Thus
    \[
    \sum_{\sigma\in S_n}T(\sigma)=\sum_{F\in\F}|F|!(n-|F|)!.
    \]
    Dividing both sides by $n!$, we get
    \[
    \Lambda(n,\F)=\frac{\sum_{\sigma\in S_n}T(\sigma)}{n!}=\frac{\sum_{F\in\F}|F|!(n-|F|)!}{n!}=\sum_{F\in\F}\frac{1}{\binom{n}{|F|}}.\qedhere
    \]
\end{proof}

As we will see in \Cref{sec:exact}, this proof is very closely related to Lubell's proof of the BLYM inequality, which relies on the observation that if $\F$ contains no comparable pairs, then $\Lambda(n,\F)\le 1$.
Like in Lubell's proof, \Cref{Lubell function} can be used to upper bound the size of $\F$ as follows.

\begin{lemma}
\label{equivalence}
    For all $n\ge 1$ and all $1\le \comporder\le 2^n$,
    \(
    \La(n,\comporder)\le\Lambda^\star(n,\comporder) \La(n). 
    \)
\end{lemma}

\begin{proof}

Notice that the righthand-side  in \Cref{Lubell function} is minimised if $|F|=\lfloor n/2\rfloor$ for all $F\in\F$, giving 
\[
\Lambda(n,\F)\ge \frac{|\F|}{\binom{n}{\lfloor n/2\rfloor}}
\quad\text{ or, equivalently, }\quad
|\F|\le\Lambda(n,\F)\La(n) 
\]
for all $\F\subseteq 2^{[n]}$. 
In particular, if $G_\F$ has components of order at most $\comporder$, then
\(
|\F|\le\Lambda^\star(n,\comporder)\La(n).  
\)
\end{proof}

So for \Cref{main} it is sufficient to show that $\Lambda^\star(n,\comporder)\le 1+o(1)$.

\subsection{Towards \Cref{main}}\label{sub:reduction}

Before proving \Cref{main}, we first handle the simpler case where $\comporder$ is fixed, as the proof strategy is the same, but the details are easier to follow. The full proof of \Cref{main} requires a more careful analysis and fine control on the edge density of a subfamily on two consecutive layers. 

For a given family $\F\subseteq 2^{[n]}$, it will be helpful to have a certain basic property.
Let us call a set $Y\in 2^{[n]}$ a {\em skip} of $\F$ if $Y\notin \F$ and there exist $X,Z\in \F$ such that $X\subseteq Y\subseteq Z$.
We say that $\F$ is {\em skipless} if it admits no skip.
In the context of computing $\La(n,\comporder)$, we may assume skiplessness without loss of generality,
 by the following claim.

\begin{lemma}
\label{skipless}
    For any $\comporder>0$, there exists a skipless family $\F\subseteq 2^{[n]}$ of size  $\La(n,\comporder)$ such that  the components of $G_\F$ all have order at most $\comporder$. 
\end{lemma}

\begin{proof}
Let $\F$ be a family of size $\La(n,\comporder)$ such that the components of $G_\F$ all have order at most $\comporder$, and moreover assume that $\F$ is such that it has the fewest skips. For a contradiction, let us assume that $\F$ is not skipless. That is, there is some set $Y\in 2^{[n]}$ such that $Y\notin \F$ and there exist $X,Z\in\F$ such that $X\subseteq Y\subseteq Z$.

Let us first consider the family $\F'$ formed by adding $Y$ to  $\F$.
In $\F'$  clearly $Y$ is not a skip.
Moreover, $Y$ does not certify a skip in $\F'$ that was not a skip in $\F$, since $X\subseteq Y\subseteq Z$.
Thus $\F'$ has fewer skips than $\F$ has.

Let $C$ be the component of $G_\F$ that contains $X$ and $Z$. Let $C'$ be the component of $G_{\F'}$ that contains $X$, $Z$ and, necessarily, $Y$. Since $X\subseteq Y\subseteq Z$, we know that $C'$ is equal to $C$ with $Y$ added.
Let $X_{\max}$ be the maximal element of $C$, which is also the maximal element of $C'$.

Next consider the family $\F''=\F'\setminus \{X_{\max}\}$ of cardinality $|\F\cup \{Y\}\setminus \{X_{\max}\}|=|\F|=\La(n,\comporder)$. The component of $G_{\F''}$ containing $Y$ must be $C\cup \{Y\} \setminus \{X_{\max}\}$ and has cardinality $|C|$. We thus conclude that the components of $G_{\F''}$ all have order at most $\comporder$.
The removal of $X_{\max}$ from $\F'$ cannot produce a skip, and so $\F''$ has fewer skips than $\F$, a contradiction to the choice of $\F$.
\end{proof}

Note that this only says there is at least one skipless family that is optimal, but does not imply that all optimal families are skipless (although that might be true).

In our argument, we find it convenient to isolate our consideration of sets in the family that are not too large or too small, so we introduce a third parameter in the definitions of $\La(n,\comporder)$ and $\Lambda^\star(n,\comporder)$ to indicate this constraint. Let $\La(n,\comporder,k)$ denote the size of a largest family $\F\subseteq 2^{[n]}$ such that all the components of $G_\F$ have order at most $\comporder$ and such that $k\le |F|\le n-k$ for all $F\in\F$. Similarly $\Lambda^\star(n,\comporder,k)$ is the maximum of $\Lambda(n,\F)$ over all such families $\F$. The analogue of \Cref{equivalence} holds:

\begin{equation}
\label{modified equiv}
\La(n,\comporder,k)\le\Lambda^\star(n,\comporder,k) \La(n). 
\end{equation}
Moreover, the analogue of \Cref{skipless} for $\La(n,\comporder,k)$ holds by the same proof: there is a skipless family $\F$ of size $\La(n,\comporder,k)$, no member of which has more than $k$ or fewer than $n-k$ elements, such that the components of $G_\F$ all have order at most $\comporder$.

Before proceeding, we require some more notation.
    For $\F\subseteq 2^{[n]}$, we denote by $S(\F)$ the set of permutations that meet $\F$ at least once and let $\Ssize(\F)=|S(\F)|$. Furthermore, for $i=0,\dots,n$, denote by $S_i(\F)$ the set of permutations that meet $\F$ exactly $i$ times and let $\Ssize_i(\F)=|S_i(\F)|$. So $\Ssize(\F)=\Ssize_1(\F)+\dots+\Ssize_n(\F)$.
    Finally, denote by $\overline{\Lambda}(n,\F)$ the average number of times a permutation $\sigma\in S(\F)$ meets $\F$, that is,
    \[
    \overline{\Lambda}(n,\F)=\frac{\sum_{i=1}^ni\Ssize_i(\F)}{\Ssize(\F)}.
    \]
    Note that $\overline{\Lambda}(n,\F)\ge\Lambda(n,\F)$.

The reason for this finer notation is that it will be helpful to partition $\F$ and consider the average over the permutations meeting the specific parts separately, by the following lemma.
\begin{lemma}
\label{partitioning}
    Given a  partition $\F_1\cup \dots\cup \F_k$ of $\F\subseteq 2^{[n]}$,
    \[
    \Lambda(n,\F)\le\overline{\Lambda}(n,\F)\le\max_{1\le i\le k}\overline{\Lambda}(n,\F_i).
    \]
\end{lemma}

\begin{proof}
    This is a general fact: if we partition a set of positive numbers, then the average of all the numbers cannot be larger than the maximum of the averages of the parts.
\end{proof}

\noindent
Generally speaking, the more parts, the easier the bound on $\overline{\Lambda}(n,\F_i)$ is to compute, but the weaker the resultant bound in \Cref{partitioning}. For the next proof, it suffices to partition $\F$ into singletons $\{\{F\}\}_{F\in\F}$, but later for \Cref{main} in full we will need to consider larger parts.

\begin{proof}[Proof of \Cref{main} if $\comporder>0$ is fixed]
    We first show that $\La(n,\comporder,\log n)\le (1+o(1))\La(n)$, and 
    then at the end derive the result by simply adding in all the sets of size less than $\log n$ or more than $n-\log n$. The choice of $\log n$ here is not canonical, as any slowly increasing function in $n$ will do.
    
    By \Cref{equivalence}, or rather \Cref{modified equiv}, it is sufficient to show that $\Lambda^\star(n,\comporder,\log n)\le 1+o(1)$ as $n\to\infty$.
    Let $\F$ be a family of size $\La(n,\comporder,\log n)$, no member of which has more than $n-\log n$ or fewer than $\log n$ elements, such that the components of $G_\F$ all have order at most $\comporder$.
    We want then to show that  $\Lambda(n,\F)\le 1+o(1)$ (where the $o(1)$ term depends only on $n$ and not on the family $\F$).
 By \Cref{skipless}, or rather its $\La(n,\comporder,k)$ analogue, we may assume that $\F$ is skipless.
    
    We actually show something slightly stronger than we need, that for each set $F\in\F$ we have $\overline{\Lambda}(n,\{F\})\le 1+o(1)$ as $n\to\infty$. By \Cref{partitioning}, we will then have $\Lambda(n,\F)\le\max_{F\in\F}\overline{\Lambda}(n,\{F\})\le 1+o(1)$.
  To do this, our goal is to show that for any given set $F$ in component $C$, most permutations meeting $F$ only meet $C$ once. So consider the set  $S(\{F\})$ of permutations meeting $F$. We partition $S(\{F\})$ into parts according to the valuation of the pair $(\sigma(|F|),\sigma(|F|+1))$. There are $|F|(n-|F|)$ such parts, since $F$ has $|F|$ elements that can determine the first coordinate and $n-|F|$ non-elements that can determine the second coordinate. These parts are all disjoint and of equal size, namely, all having $(|F|-1)!(n-|F|-1)!$ permutations. 
    
    Since $\F$ is skipless, any permutation $\sigma \in S(\{F\})$ meets $C$ (and thus $F$) only once if and only if both sets $\sigma([|F|-1])$ and  $\sigma([|F|+1])$ are not in $C$.
    
    Writing $M_1$ for the number of subsets of $F$ in $C$, at most $(n-|F|)M_1$ of the parts of the partition of $S(\{F\})$ contain permutations $\sigma$ with $\sigma([|F|-1])\in C$. Similarly, writing $M_2$ for the number of supersets of $F$ in $C$, at most $|F|M_2$ of the parts contain permutations $\sigma$ with $\sigma([|F|+1])\in C$. But since $|C|\le \comporder$, we have $M_1+M_2\le \comporder$. So at most $M_1(n-|F|)+M_2|F|\le \comporder\max(|F|,n-|F|)$ of the parts contain permutations that meet $C$ more than once, that is,
    \begin{equation}
    \label{more than once}
    \sum_{i=2}^\comporder\Ssize_i(\{F\})\le \comporder\max(|F|,n-|F|)(|F|-1)!(n-|F|-1)!.
    \end{equation}
    We also know that $\Ssize(\{F\})=|F|!(n-|F|)!$. Hence
    \begin{align*}
    \overline{\Lambda}(n,\{F\})
    &=\frac{\Ssize_1(\{F\})+\sum_{i=2}^\comporder i\Ssize_i(\{F\})}{|F|!(n-|F|)!} 
    \le 1+\frac{\comporder^2\max(|F|,n-|F|)(|F|-1)!(n-|F|-1)!}{|F|!(n-|F|)!}\\
    &=1+\frac{\comporder^2\max(|F|,n-|F|)}{|F|(n-|F|)}
    =1+\frac{\comporder^2}{\min(|F|,n-|F|)}\le 1+\frac{\comporder^2}{\log n}\le 1+o(1).
    \end{align*}
    In the first inequality, we used \Cref{more than once} and the fact that $\Ssize_1(\{F\})\le \Ssize(\{F\})$. In the penultimate inequality, we used the fact that $\log n\le|F|\le n-\log n$, and at the end we used that $\comporder$ is fixed.

    By \Cref{partitioning} and \Cref{modified equiv}, this suffices to show
    \(
    \La(n,\comporder,\log n)\le (1+o(1))\La(n) 
    \)
     as $n\to\infty$.
    It only remains to account for the constraint on the size of $F$. But for this, we can just observe that the number of sets we excluded is negligible, specifically, that
    \begin{align*}
    \La(n,\comporder)
    &\le \La(n,\comporder,\log n)+|\{X\subseteq 2^{[n]}:|X|<\log n\text{ or }|X|>n-\log n\}|\\
    &\le (1+o(1))\La(n)+2\log n\cdot \binom{n}{\lfloor\log n\rfloor}\le (1+o(1))\La(n) 
    \end{align*}
    as $n\to\infty$, where the last inequality follows for example by Stirling's approximation.
\end{proof}

That there is room for improvement in the estimates we made is evident. For example, we bounded the entire average by separately bounding the average for each singleton set in $\F$, which clearly is not optimal. 
Instead of singletons, we can take the parts as components of $G_\F$. If we believe \Cref{main conjecture} and the correctness of our proposed extremal construction, it is intuitive that all components are (near) equal in size in the optimal case, and so we should not lose much by doing this. In other words, we believe there is an optimal family $\F$ such that $\Lambda(n,\F)$ is equal to $\overline{\Lambda}(n,C)$ for all components $C$.

Another important step is where we upper bounded, for a given set $F\in C$, the number of subsets and supersets of $F$ that are also in $C$ in the layers immediately above and below that of $F$. Thanks to the assumption that $\F$ is skipless, this shows that only a small fraction of the permutations meeting $F$ also meet other sets in $\F$. We have crudely bounded this number of subsets and supersets by $\comporder$. When $\comporder$ is a fixed constant this is good enough, but if $\comporder(n)$ grows quickly it is far too loose. 
Indeed, if this bound were near to optimal on average, it would mean that the component $C$ has approximately $\comporder^2/2$ edges, which we will see is impossible.

Being more careful, we are in fact interested in a bound on  the average degree of a vertex of $G'_\F$. It turns out that a good general bound on the  edge density between two consecutive layers on the Boolean lattice will suffice for our purposes.
Let us define some notation for this task.
    For $0\le k< n$, let $\A\subseteq\binom{[n]}{k}$, $\B\subseteq\binom{[n]}{k+1}$. Then we denote by $\xi(\A,\B)$ the number of comparable pairs in $\A\cup \B$.
    We also denote by $\xi^\star_n(m)$ the function that, given $m>0$, returns the maximum of $2\xi(\A,\B)/m$ over all pairs $(\A,\B)$ with $\A\subseteq\binom{[n]}{k}$, $\B\subseteq\binom{[n]}{k+1}$ and $|\A|+|\B|=m$ over all $0\le k<n$.
That is, $\xi^\star_n(m)$ is the largest average degree of a subgraph of order $m$ on any two consecutive layers of $G_{2^{[n]}}$.
Using a refinement of the argument above, we next show how control on $\xi^\star_n$ is enough to deduce \Cref{main}.

\begin{lemma}
\label{main lemma}
Let $\comporder=\comporder(n)$. 
As $n\to\infty$, if $\xi^\star_n(\comporder)=o(n)$, then
     $\La(n,\comporder) \le (1+o(1)) \La(n)$. 
\end{lemma}

\begin{proof}
    We shall always assume that $n$ is large enough when needed. 
    We will first bound $\La(n,\comporder,n/4)$ and then at the end note that we left out a negligible number of sets.
    The choice of $n/4$ is unimportant, as any $rn$ with fixed $0<r<1/2$ suffices.
    By \Cref{modified equiv} it is sufficient to show that $\Lambda^\star(n,\comporder,n/4)\le 1+o(1)$ as $n\to\infty$. 
    Let $\F$ be a family of size $\La(n,\comporder,n/4)$, no member of which has more than $3n/4$ or fewer than $n/4$ elements, such that the components of $G_\F$ all have order at most $\comporder$.
    We want then to show that  $\Lambda(n,\F)\le 1+o(1)$ (where the $o(1)$ term depends only on $n$ and not on the family $\F$).
 By \Cref{skipless}, or rather its $\La(n,\comporder,k)$ analogue, we may assume that $\F$ is skipless.

    We actually show $\overline{\Lambda}(n,C)\le 1+o(1)$  for each component $C$ of $G_\F$. By \Cref{partitioning}, this then gives
    \[
    \Lambda(n,\F)\le\max_{C\text{ component of }G_\F}\overline{\Lambda}(n,C)\le 1+o(1).
    \]
    Now call a permutation {\em bad} if it meets $\F$ more than once, and {\em good} otherwise. Our goal is to upper bound the number of bad permutations.
    
    For a given component $C$ and some $k\in[n-1]$, denote by $C_k$ the set of elements of $C$ on layer $k$ and recall that $S(C_k)$ is the set of permutations that have an initial segment in $C_k$, and also recall that $\Ssize(C_k)=|S(C_k)|$. We then partition $S(C_k)$ according to the valuation of the triple $(\sigma([k]),\sigma(k),\sigma(k+1))$ (where the first coordinate indeed indexes the element of $C_k$). This partitions $S(C_k)$ into $|C_k|k(n-k)$ parts. 
   These parts are disjoint and of equal size, namely, with $(k-1)!(n-k-1)!$ permutations.
   
       Since $\F$ is skipless, any permutation $\sigma \in S(C_k)$ meets $C_k$ only once if and only if both sets $\sigma([k-1])$ and  $\sigma([k+1])$ are not in $C$.
   Thus, by the definition of the parts, each part contains either only bad permutations or only good permutations. Moreover, every bad permutation must meet both endpoints of some edge in $G_\F$ between $C_{k-1}$ and $C_k$ or between $C_k$ and $C_{k+1}$. This is depicted in \Cref{G}.
    
    \begin{figure}
    \centering
\definecolor{ududff}{rgb}{0.30196078431372547,0.30196078431372547,1}
\definecolor{uuuuuu}{rgb}{0,0,0}
\begin{tikzpicture}[line cap=round,line join=round,>=triangle 45,x=1cm,y=1cm]
\clip(5.44,6) rectangle (17.88,11.5);
\draw [line width=2pt] (8,11)-- (9,7);
\draw [line width=2pt] (15,7)-- (16,11);
\draw [line width=2pt] (15.250588235294117,8.00235294117647)-- (8.751764705882353,7.992941176470588);
\draw [line width=2pt] (8.494117647058824,9.023529411764706)-- (15.498823529411766,8.99529411764706);
\draw [line width=2pt] (15.754117647058823,10.016470588235293)-- (8.245882352941177,10.016470588235293);
\draw [line width=2pt] (9.339919865896785,9.02012006456079)-- (8,8);
\draw [line width=2pt] (8,8)-- (7.16,7.04);
\draw [line width=2pt] (8,8)-- (8.8,6.96);
\draw [line width=2pt] (12.000037879954169,9.009397385296872)-- (11.326598055516463,7.996670116875562);
\draw [line width=2pt] (11.326598055516463,7.996670116875562)-- (10.4,7);
\draw [line width=2pt] (11.326598055516463,7.996670116875562)-- (12,7);
\draw [line width=2pt] (12.000037879954169,9.009397385296872)-- (12,10.016470588235293);
\draw [line width=2pt] (12,10.016470588235293)-- (11.28,10.96);
\draw [line width=2pt] (12,10.016470588235293)-- (12.72,10.98);
\draw (11.82,11.02) node[anchor=north west] {$\dots$};
\draw (6.62,11.00) node[anchor=north west] {$\dots$};
\draw (7.8,7.28) node[anchor=north west] {$\dots$};
\draw (10.7,7.16) node[anchor=north west] {$\dots$};
\draw [line width=2pt] (9.339919865896782,9.020120064560789)-- (6.82,10.02);
\draw [line width=2pt] (6.82,10.02)-- (6.06,10.96);
\draw [line width=2pt] (6.82,10.02)-- (7.52,10.94);
\draw [line width=2pt] (14.799992385778875,8.998111038615685)-- (14.560000000000107,10.016470588235293);
\draw [line width=2pt] (14.560000000000107,10.016470588235293)-- (13.96,10.96);
\draw [line width=2pt] (14.560000000000107,10.016470588235293)-- (15.26,10.96);
\draw [line width=2pt] (14.799992385778875,8.998111038615685)-- (16.22,8);
\draw [line width=2pt] (16.22,8)-- (15.72,7.04);
\draw [line width=2pt] (16.22,8)-- (16.98,7.08);
\draw (15.92,7.2) node[anchor=north west] {$\dots$};
\draw (14.4,11.04) node[anchor=north west] {$\dots$};
\draw [line width=2pt] (6.82,10.02)-- (6.42,10.94);
\draw [line width=2pt] (8,8)-- (7.54,7.04);
\draw [line width=2pt] (12,10.016470588235293)-- (11.62,10.94);
\draw [line width=2pt] (11.326598055516461,7.996670116875562)-- (11.68,7);
\draw [line width=2pt] (14.56,10.016470588235293)-- (14.26,10.96);
\draw [line width=2pt] (16.273797752809042,7.934876404494317)-- (16.72,7.08);
\begin{scriptsize}
\draw[color=black] (13.35,8.20) node {$C_{k-1}$};
\draw[color=black] (13.28,9.29) node {$C_k$};
\draw[color=black] (13.29,10.21) node {$C_{k+1}$};
\end{scriptsize}
\end{tikzpicture}
\vspace{-15pt}
    \caption{$S(C_k)$ is split into $|C_k|k(n-k)$ parts, three of which are depicted here. The leftmost part contains only good permutations, the other two contain only bad permutations.}
    \label{G}
\end{figure}
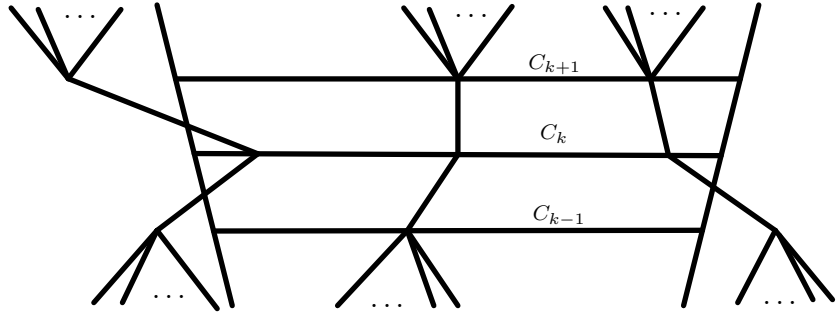
    
    By assumption, there are at most 
    \[
    \frac{1}2(|C_k|+|C_{k+1}|)\xi^\star_n(|C_k|+|C_{k+1}|)\le \frac{1}2(|C_k|+|C_{k+1}|)\xi^\star_n(\comporder)   
    \]
    edges between $C_k$ and $C_k+1$ and similarly at most 
    \[
    \frac{1}2(|C_k|+|C_{k-1}|)\xi^\star_n(|C_k|+|C_{k-1}|)\le \frac{1}2(|C_k|+|C_{k-1}|)\xi^\star_n(\comporder)   
    \]
    edges between $C_{k-1}$ and $C_k$. Thus at most 
    \[
    \frac{1}2(|C_k|+|C_{k-1}|)\xi^\star_n(\comporder)(n-k)+\frac{1}2(|C_k|+|C_{k+1}|)\xi^\star_n(\comporder)k 
    \]
    of the parts of $S(C_k)$ contain bad permutations. These account for at most 
    \begin{align*}
    &\frac{1}2\xi^\star_n(\comporder)((|C_k|+|C_{k-1}|)(n-k)+(|C_k|+|C_{k+1}|)k)(k-1)!(n-k-1)! \\
    &\le \frac{1}2\xi^\star_n(\comporder)(|C_{k-1}|+|C_k|+|C_{k+1}|)n(k-1)!(n-k-1)! 
    \end{align*}
    bad permutations.
    
    As the above argument holds true for each layer $C_k$, $k\in[n-1]$, we have at most
    \begin{align}
    \label{bad perm bound}
    \frac{1}2\xi^\star_n(\comporder)\sum_{k=1}^{n-1}(|C_{k-1}|+|C_k|+|C_{k+1}|)n(k-1)!(n-k-1)! 
    \end{align}
    bad permutations per component $C$. But actually we have fewer, because if a permutation $\sigma$ meets $C$ $i$ times, then we are counting it $i$ times in our sum, once for each layer it meets. So the sum above is actually an upper bound for
    \[
    \sum_{i=2}^ni\Ssize_i(C).
    \]
    On the other hand, we have 
    \begin{equation}
    \label{all perm bound}
    \Ssize(C)=\frac{\sum_{k=1}^{n-1}|C_k|k!(n-k)!}{\overline{\Lambda}(n,C)}.
    \end{equation}
    This is by definition of $\overline{\Lambda}$, since
    \[
    \sum_{k=1}^{n-1}|C_k|k!(n-k)!
    \]
    counts all permutations meeting $C$ with repetition, i.e.~it is equal to $\sum_{i=1}^ni\Ssize_i(C)$ (see the proof of \Cref{Lubell function}).
    Putting things together, we obtain
    \begin{align*}
    \frac{\sum_{i=2}^ni\Ssize_i(C)}{\Ssize(C)}
    &\le\frac{n\xi^\star_n(\comporder)\sum_{k=1}^{n-1}(|C_{k-1}|+|C_k|+|C_{k+1}|)(k-1)!(n-k-1)!}{2\sum_{k=1}^{n-1}|C_k|k!(n-k)!/\overline{\Lambda}(n,C)}\\ 
    &\le\frac{n\xi^\star_n(\comporder)\sum_{k=1}^{n-1}5|C_k|(k-1)!(n-k-1)!}{2\sum_{k=1}^{n-1}|C_k|k!(n-k)!/\overline{\Lambda}(n,C)}\\ 
    &\le\frac{40\xi^\star_n(\comporder)\sum_{k=1}^{n-1}|C_k|k!(n-k)!}{n\sum_{k=1}^{n-1}|C_k|k!(n-k)!/\overline{\Lambda}(n,C)} 
    =\frac{40\overline{\Lambda}(n,C)\xi^\star_n(\comporder)}n. 
    \end{align*}
    The first inequality is from \Cref{bad perm bound,all perm bound}. In the second inequality, we regrouped the sum by each $|C_k|$, so that the total coefficient for each term is
    \[
    (k-2)!(n-k)!+(k-1)!(n-k-1)!+k!(n-k-2)!\le 5(k-1)!(n-k-1)!.
    \]
    This is true because either one of $(k-2)!(n-k)!$ and $k!(n-k-2)!$ is smaller than $(k-1)!(n-k-1)!$ while the other one is at most $3(k-1)!(n-k-1)!$, provided that $n/4\le k\le 3/4n$, which we have assumed to be the case. In the third inequality, we multiplied the numerator by $16k(n-k)$ and the denominator by $n^2$. We may do this as long as $16k(n-k)\ge n^2$, which is true since $k\ge n/4$ and $n-k\ge n/4$. 
    
    Now we have
    \[
    \overline{\Lambda}(n,C)=\frac{\Ssize_1(C)}{\Ssize(C)}+\frac{\sum_{i=2}^ni\Ssize_i(C)}{\Ssize(C)}\le 1+\frac{40\overline{\Lambda}(n,C)\xi^\star_n(\comporder)}n. 
    \]
    If we rearrange, we get
    \[
    \overline{\Lambda}(n,C)\le\left(1-\frac{40\xi^\star_n(\comporder)}n\right)^{-1}. 
    \]

By the assumption that $\xi^\star_n(\comporder)=o(n)$, we deduce that $\overline{\Lambda}(n,C)\le 1+o(1)$. Since this is true of each component of $\F$, it follows that $\Lambda(n,\F)\le 1+o(1)$ by \Cref{partitioning}. Since this is true for all families $\F$  of size $\La(n,\comporder,n/4)$ with components of order at most $\comporder$ and sets of sizes between $n/4$ and $3/4n$, we deduce that $\Lambda^\star(n,\comporder,n/4)\le 1+o(1)$. By \Cref{modified equiv}, we then conclude that   
\(
\La(n,\comporder,n/4)\le (1+o(1))\La(n) 
\)
as $n\to\infty$.
    It only remains to account for the constraint on the size of $F$. But for this, we can just observe that the number of sets we excluded is negligible, specifically, that
\begin{align*}
\La(n,\comporder)
&\le\La(n,\comporder,n/4)+|\{X\subseteq 2^{[n]}:|X|<n/4\text{ or }|X|>3/4n\}|\\
&\le(1+o(1))\La(n)+\frac{n}{2}\binom{n}{\lfloor n/4\rfloor} \le (1+o(1))\La(n) 
\end{align*}
    as $n\to\infty$, where the last inequality follows for example by Stirling's approximation.
\end{proof}

\subsection{Bounding $\xi^\star_n$}\label{sub:xistar}

Upper bounds on $\xi^\star_n$ can be seen in interesting counterpoint to the result of Kleitman~\cite{Kle68} that minimises the number of $2$-chains.
Indeed, bounding $\xi^\star_n$ from above is a question of maximising $2$-chains, and similar questions have a long tradition; see for example~\cite{DaFr83,AlFr85,ADGS15}.
Curiously, the following question was only posed relatively recently.

\begin{question}[Question 5.1~\cite{ADGS15}]
\label{question 5.1}
Given $0\le k_1\le k_2\le n$, $0\le a\le \binom{n}{k_1}$ and $0\le b\le\binom{n}{k_2}$, which families $\A\subseteq\binom{[n]}{k_1}$ of size $a$ and $\B\subseteq\binom{[n]}{k_2}$ of size $b$ maximise $\xi(\A,\B)$?
\end{question}

\noindent
Maximising $\xi^\star_n(m)$ corresponds to optimising this question over all $k_1,k_2,a,b$ satisfying $k_2=k_1+1$ and $a+b=m$.
As the authors of~\cite{ADGS15} point out, there is an interesting connection between \Cref{question 5.1} with $k_2=k_1+1$ and the Kruskal--Katona theorem (\Cref{thm:KrKa}) in the case $r=1$. Indeed we can infer that the minimum possible size of ${\mathcal S}_1$ is such that $|\F|k\le\frac{1}2(|\F|+|{\mathcal S}_1|)\xi^\star_n(|\F|+|{\mathcal S}_1|)$, since $|\F|k$ is the number of edges between $\F$ and ${\mathcal S}_1$, which by definition of $\xi^\star_n$ has to be less than $\frac{1}2(|\F|+|{\mathcal S}_1|)\xi^\star_n(|\F|+|{\mathcal S}_1|)$. So an upper bound on $\xi^\star_n$ translates into a lower bound on $|{\mathcal S}_1|$, potentially giving a weaker version of Kruskal--Katona.

As we showed in \Cref{sub:reduction},
\Cref{main lemma} allows us to use results about $\xi^\star_n$ to obtain results about $\La(n,\comporder)$. We now relate $\xi^\star_n$ to a well-studied Tur\'an-type problem.
Recall that a proper edge-colouring of a graph $G=(V,E)$ is a mapping $c: E\to \Z^+$ such that any two incident edges $e_1, e_2$ must have $c(e_1)\ne c(e_2)$.
A cycle in $G$ of length $\ell\ge 3$ is called {\em rainbow} (with respect to $c$) if no two edges of the cycle receive the same colour.
For $\ell\ge3$, we denote by $\mad^\star(\comporder,\mathcal{O}_\ell)$, the largest average degree in a graph $G$ of order $\comporder$ such that there exists a proper edge-colouring of $G$ with no copies of a rainbow cycle of length $\ell$.
We denote by $\mad^\star(\comporder,\mathcal{O}^\star)$ the largest average degree in a graph $G$ of order $\comporder$ such that there exists a proper edge-colouring of $G$ with no rainbow cycles.

%
%
%
\begin{lemma}
\label{rainbow to lattice}
    For all $n,\comporder\ge 1$ we have
    \(
    \xi^\star_n(\comporder)\le \mad^\star(\comporder,\mathcal{O}^\star).
    \)
\end{lemma}

\begin{proof}
    Assume for a contradiction that the inequality does not hold. This means there are two sets $\A\subseteq\binom{[n]}{k}$, $\B\subseteq\binom{[n]}{k+1}$ for some $0\le k<n$ with $|\A|+|\B|=\comporder$ and $2\xi(\A,\B)/\comporder>\mad^\star(\comporder,\mathcal{O}^\star)$. By definition of $\mad^\star(\comporder,\mathcal{O}^\star)$, this means that every proper edge-colouring of the graph $G_{\A\cup \B}$ contains a rainbow cycle. However, let us consider the following edge-colouring: given any edge $e$ of $G_{\A\cup\B}$, and supposing it corresponds to the pair $(A, A\cup\{i\})$ for some $i\in [n]$, then colour $e$ with colour $i$.

    The edge-colouring is clearly proper, since if $X$ had two incident edges with colour $i$, these would both have as other endpoint $X\cup\{i\}$ (or $X\backslash\{i\}$ depending on whether $X\in\A$ or $X\in\B$). Also, the colouring contains no rainbow cycles: consider any cycle and any starting point $A\in\A$ for a traversal of the cycle (in either direction). The first edge will have some colour $i\in [n]$, corresponding to its other endpoint $A\cup\{i\}$. Now as we continue the transversal, the element $i$ will eventually have to be removed at some point, at which time the corresponding edge will have colour $i$ (and it is necessarily distinct from the edge corresponding to $(A,A\cup \{i\})$). So the cycle has colour $i$ at least twice, and so is not a rainbow cycle. We have arrived at a contradiction, as desired.
\end{proof}

Notice that the above argument in fact gives some handle on \Cref{question 5.1} in the specific case $k_2=k_1+1$, but it moreover breaks down for this purpose when $k_2>k_1+1$.

\Cref{rainbow to lattice,main lemma} gives us bounds on $\La(n,\comporder)$ via the establishment of good bounds on $\mad^\star(\comporder,\mathcal{O}^\star)$. There has already been much work on this latter problem. This was initiated by Keevash, Mubayi,
Sudakov and Verstra\"{e}te~\cite{KMSV07}, who showed that $\mad^\star(\comporder,\mathcal{O}^\star)=O(\comporder^{1/3})$, simply by bounding it with $\mad^\star(\comporder,\mathcal{O}_6)$.
There have been numerous successive improvements~\cite{DLS13,Jan23,Tom24,KLLT24,JaSu24}. The current state of the art is due to Alon, Buci\'c, Sauermann, Zakharov, Zamir~\cite{ABSZZ+}.

\begin{theorem}[\cite{ABSZZ+}]
\label{Alon}
$\mad^\star(\comporder,\mathcal{O}^\star)=O(\log \comporder \log\log \comporder)$.
\end{theorem}

\noindent
This result is sharp up to at most a $O(\log\log \comporder)$ factor. They proved this result with a sophisticated argument based on robust sublinear expanders.
\Cref{main} follows immediately from \Cref{Alon}.

\begin{proof}[Proof of \Cref{main}]
The lower bound on $\La(n,\comporder)$ follows from considering $\F = \binom{[n]}{\lfloor n/2 \rfloor}$, in which case $G_\F$ has only singleton components.
By \Cref{rainbow to lattice} and \Cref{Alon}, $\xi^\star_n(\comporder)=O(\log \comporder\log\log \comporder)$. Since $\comporder=O(2^{g(n)})$ with $g(n)=o(n/\log n)$ as $n\to\infty$, which is the premise of \Cref{main}, it follows that $\log \comporder\log\log \comporder=O(g(n)\log g(n))=o(n)$. And so we conclude from \Cref{main lemma} that 
    \(
    \La(n,\comporder)\le (1+o(1))\La(n). 
    \)
\end{proof}

\section{Exact approaches}\label{sec:exact}

We now switch to a discussion about exact results relating to \Cref{main conjecture}. These include in \Cref{sub:BLYM} a generalisation of the BLYM inequality and a proof that \Cref{main conjecture} holds if the components have a specific shape, and in \Cref{sub:exact} the cases $k=1$ and $k=2$ of \Cref{main conjecture}.

\subsection{Generalising the BLYM inequality}\label{sub:BLYM}

In order to tackle some special cases of \Cref{main conjecture}, we employ a modified version of the BLYM inequality (\Cref{thm:BLYM}) which allows for the presence of $2$-chains. The problem then is that it becomes more difficult to precisely count each permutation when we add up over the sets of $\F$. In \Cref{sec:proof}, we showed that on average we count each permutation $1+o(1)$ times, but the error term, while negligible, makes this inexact. On the other hand, if we assume a particular form for each the components, then we can count the permutations precisely. 

Let us call a family $\F'\subseteq 2^{[n]}$ a {\em diamond} if it consists of all sets $X$ such that $A\subseteq X\subseteq B$ for some sets $A\subseteq B \subseteq [n]$, which we refer to as the bottom and top corners, respectively, of the diamond. 

The following result may be of independent interest.
\begin{theorem}[Diamond BLYM]
\label{diamond BLYM}
For  $\F\subseteq 2^{[n]}$, suppose that the components of $G_\F$ are induced by diamond subfamilies $\F_1,\dots,\F_\ell$. 
For $i,j\in \{0,\dots,n\}$, let $a_{i,j}$ denote the number of components of $G_\F$ whose bottom corner is on layer $i$ and whose top corner is on layer $i+j$. Then
\begin{equation}
\label{BLYM}
\sum_{i=0}^n\sum_{j=0}^n\frac{a_{i,j}}{\binom{n-j}i}\le 1.
\end{equation}
\end{theorem}

\begin{proof}
    We count the permutations of $[n]$ in two different ways. On the one hand there are $n!$ permutations. On the other hand, we can count the number $\Ssize(\F)$ of permutations $\sigma$ that meet $\F$. Notice that the permutations counted for each component are distinct, since if $X\in\F_i$ and $Y\in\F_j$ are both initial segments of the same permutation, then $X\subseteq Y$ or $Y\subseteq X$, thus $i=j$. So 
    \[
    \sum_{k=1}^\ell \Ssize(\F_k)\le n!.
    \]
    Now say $\F_k$ has bottom corner $A_k$ and top corner $B_k$. Then notice that a permutation $\sigma$ contains as initial segment some set $X\in\F_k$ if and only if all the elements of $A_k$ appear before all the elements of $[n]\backslash B_k$ in $\sigma$. Indeed, if this is the case, then all the initial segments of $\sigma$ whose last element comes after or at the last element of $A_k$ and before the first element of $[n]\backslash B_k$ are sets in $\F_k$ since they are supersets of $A_k$ and subsets of $B_k$. On the other hand, if not all elements of $A_k$ come before all elements of $[n]\backslash B_k$, then no initial segment of $\sigma$ can be a set in $\F_k$, since the last element of the segment would have to be after or at the last element of $A_k$, but then it would also contain some element of $[n]\backslash B_k$, which would imply that it cannot be a subset of $B_k$. 
    
    With this, we can simply count $\Ssize(\F_k)$ as
    \[
    \frac{n!}{(n-(|B_k|-|A_k|))!}|A_k|!(n-|B_k|)!
    \]
    (i.e.~the elements of $B_k\backslash A_k$ can be at any position, but then the positions for the elements of $A_k$ are fixed, apart from their order, and the same holds for $[n]\backslash B_k$). Hence we get
    \begin{align*}
    \sum_{k=1}^\ell \Ssize(\F_k)
    =\sum_{k=1}^\ell \frac{n!|A_k|!(n-|B_k|)!}{(n-(|B_k|-|A_k|))!}
    =\sum_{i=0}^n\sum_{j=0}^n a_{i,j}\frac{n!i!(n-i-j)!}{(n-j)!}
    \le n!.
    \end{align*}
    By dividing both sides of this inequality by $n!$, we get
    \[
    \sum_{i=0}^n\sum_{j=0}^n a_{i,j}\frac{i!(n-i-j)!}{(n-j)!}
    \le 1,
    \]
    as required.
\end{proof}

From this, it follows that \Cref{main conjecture} holds under the additional assumption that each component is a diamond.

\begin{theorem}
\label{all diamond}
Let $k,n$ be integers with $0\le k\le n$. 
If for some $\F\subseteq 2^{[n]}$ the components of $G_\F$ are diamonds of height at most $k$, then
    \[
    |\F|\le 2^k\binom{n-k}{\lfloor (n-k)/2\rfloor}.
    \]
\end{theorem}

\begin{proof}
    We can rewrite \Cref{BLYM} as follows:
    \[
    \sum_{j=0}^k\sum_{i=0}^n\frac{2^ja_{i,j}}{2^j\binom{n-j}i}\le 1.
    \]
    Note that we have 
    \[
    |\F|=\sum_{j=0}^k\sum_{i=0}^n 2^ja_{i,j}.
    \]
    Thus, to make the family as large as possible, we want all the elements of $\F$ in terms with the largest denominator. The expression in the denominator is maximised for $j=k$ and $i=\lfloor(n-k)/2\rfloor$, giving
    \[
    \frac{|\F|}{2^k\binom{n-k}{\lfloor(n-k)/2\rfloor}}\le 1. \qedhere
    \]
\end{proof}

\subsection{The cases $k=1$ and $k=2$}\label{sub:exact}

It is intuitive that the diamond shape is optimal for maximising the size of $\F$ subject to a bound on component sizes in $G_\F$, at least when $k$ and $n$ have the same parity (so that the diamond shape is exactly centred). This is because the diamond shape is very compact: it minimises the number of sets in the middle layer that are unavailable for other components.
In general, it is tricky to translate this intuition into a proof that the optimal families have to be diamond-shaped (as then we would immediately have \Cref{main conjecture} in full via \Cref{all diamond}).

This is at least the case if $k=1$, as the only possible component of order $2$ is the diamond shape. The case $k=1$ of \Cref{main conjecture} is not a new result, since having components of order at most $2$ is equivalent to $G_\F$ not containing any copies of the $\vee$-shaped graph. This was resolved by Katona and Tarj{\'a}n in 1983~\cite{KaTa83}. We note how this result also follows from the generalised BLYM inequality.

\begin{theorem}[\cite{KaTa83}]
If $\F\subseteq 2^{[n]}$ is such that $G_\F$ contains no copy of the $\vee$-shaped graph, then
\[
|\F|\le 2\binom{n-1}{\lfloor (n-1)/2\rfloor}.
\]
\end{theorem}

\begin{proof}
Since the components of $\F$ have order 1 or 2 and by \Cref{skipless} we may assume that $\F$ is skipless, the conditions of \Cref{all diamond} are satisfied for $k=1$. \Cref{all diamond} yields the result.
\end{proof}

In the case $k=2$, if $n$ is even, we can show that the diamond shape is at least as good as any other shape, when it is in the middle layers. This is enough to prove the case $k=2$ of \Cref{main conjecture}. The proof is a bit more involved and requires some case checking.

\begin{theorem}
    If $n$ is even and $\F\subseteq 2^{[n]}$ is such that $G_\F$ has components of order at most 4, then 
    \[
    |\F|\le4\binom{n-2}{n/2-1}.
    \]
\end{theorem}

\begin{proof}
    By \Cref{skipless} we may assume that $\F$ is skipless. Let $\F_1,\dots,\F_\ell$ be the components of $\F$. We start again by noting that 
    \begin{equation}
    \label{sum}
    \sum_{i=1}^\ell\Ssize(\F_i)\le n!.
    \end{equation}
    We now try to minimise $\Ssize(\F_i)$, so that we can have as many components as possible. There are a few cases we have to consider to account for different component orders and different shapes. Note that we only consider shapes that are actually different, so we avoid computing the same bound for symmetric shapes (for example the $\vee$-shape and the $\wedge$-shape). In general, for each shape we will use the inclusion--exclusion principle to calculate $\Ssize(\F_i)$, which says
    \begin{align*}
    \Ssize(\F_i)
    &=\sum_{F\in\F_i}\Ssize_1(\{F\})-\sum_{F_1,F_2\in\F_i}\Ssize_2(\{F_1,F_2\})+\sum_{F_1,F_2,F_3\in\F_i}\Ssize_3(\{F_1,F_2,F_3\})-\dots\\
    &=\sum_{F}\Ssize_1(\{F\})-\sum_{F_1\subseteq F_2}\Ssize_2(\{F_1,F_2\})+\sum_{F_1\subseteq F_2\subseteq F_3}\Ssize_3(\{F_1,F_2,F_3\})-\dots
    \end{align*}
    Then, once we obtain the expression in terms of $n$ and $k$ (the layer on which the middle of the component is located), we deduce that the expression is minimised for $k=n/2$. To do this, we use facts about binomials and simple computations of derivatives. We split the different shapes by component order:
    \begin{enumerate}[(i)]
        \item If $\F_i=\{F\}$, then $\Ssize(\F_i)=|F|!(n-|F|)!$, which is minimised when $|F|=n/2$, hence 
        \[
        \Ssize(\F_i)\ge\left(\frac{n}2\right)!\left(\frac{n}2\right)!.
        \]
        \item If $\F_i=\{F_1,F_2\}$, then, since $\F$ is skipless, we may assume that $F_1\subseteq F_2$ with $|F_2|=|F_1|+1$. Then 
        \[
        \Ssize(\F_i)=n|F_1|!(n-|F_2|)!.
        \]
        This is the case $k=1$ of the computation we did for diamond shapes in the proof of \Cref{diamond BLYM}. It is minimised for $|F_1|=n/2$ or $|F_1|=n/2-1$, giving
        \[
        \Ssize(\F_i)\ge n\cdot\left(\frac{n}2\right)!\left(\frac{n}2-1\right)!
        \]
        \item If $\F_i=\{F_1,F_2,F_3\}$ we have two options: 
        \begin{itemize}
            \item $F_1\subseteq F_2\subseteq F_3$ with $\big(|F_1|,|F_2|,|F_3|\big)=(k-1,k,k+1)$. Then
            \begin{align*}
            \Ssize(\F_i)
            &=(k-1)!(n-k+1)!+k!(n-k)!+(k+1)!(n-k-1)!-(k-1)!(n-k)!-\\
            &\qquad\qquad\qquad-k!(n-k-1)!-2(k-1)!(n-k-1)!+(k-1)!(n-k-1)!\\
            &=(k-1)!(n-k-1)!\Big(n^2-nk+k^2-1\Big)
            \ge \left(\frac{n}2-1\right)!\left(\frac{n}2-1\right)!\left(\frac{3n^2}4-1\right),
            \end{align*}
            since in the second to last line $k=n/2$ minimises both the product of the factorials (since $\binom{n-2}{k-1}$ is maximum for $k=n/2$) and the expression in the parenthesis (this can be checked by computing the derivative with respect to $k$).
            \item $F_1,F_2\subseteq F_3$ with $\big(|F_1|,|F_2|,|F_3|\big)=(k,k,k+1)$. Then
            \begin{align*}
            \Ssize(\F_i)
            &=2k!(n-k)!+(k+1)!(n-k-1)!-2k!(n-k-1)!\\
            &=k!(n-k)!+k!(n-k-1)!(n-1)\\
            &\ge \left(\frac{n}2-1\right)!\left(\frac{n}2-1\right)!\left(\frac{3}4n^2-\frac{3}2n\right),
            \end{align*}
        since both summands in the second line are minimised when $k=n/2$.
        \end{itemize}
        The latter case is slightly smaller.
        \item If $\F_i=\{F_1,F_2,F_3,F_4\}$ we have 5 cases:
        \begin{itemize}
            \item $F_1\subseteq F_2\subseteq F_3\subseteq F_4$ with $\big(|F_1|,|F_2|,|F_3|,|F_4|\big)=(k-1,k,k+1,k+2)$. Then
            \begin{align*}
            \Ssize(\F_i)
            &=(k-1)!(n-k+1)!+k!(n-k)!+(k+1)!(n-k-1)!+(k+2)!(n-k-2)!-\\
            &-(k-1)!(n-k)!-k!(n-k-1)!-(k+1)!(n-k-2)!-2(k-1)!(n-k-1)!-\\
            &-2k!(n-k-2)!-6(k-1)!(n-k-2)!+(k-1)!(n-k-1)!+k!(n-k-2)!+\\
            &+4(k-1)!(n-k-2)!-(k-1)!(n-k-2)!\\
            &=(k-1)!(n-k-2)!\left(n^3-2n^2k-n^2+2nk^2+nk-n+k^2+k-2\right)\\
            &\ge\left(\frac{n}2-1\right)!\left(\frac{n}2-1\right)!\left(n^2+\frac{3}2n+2\right).
            \end{align*}
            This follows by noticing that in the second to last line both the product of factorials and the sum in the parenthesis are minimised for $k=n/2$ (or for $k=n/2-1$ giving the same result).
            \item $F_1,F_2\subseteq F_3\subseteq F_4$, with $\big(|F_1|,|F_2|,|F_3|,|F_4|\big)=(k-1,k-1,k,k+1)$. Then
            \begin{align*}
            \Ssize(\F_i)
            &=2(k-1)!(n-k+1)!+k!(n-k)!+(k+1)!(n-k-1)!-2(k-1)!(n-k)!-\\
            &-k!(n-k-1)!-4(k-1)!(n-k-1)!+2(k-1)!(n-k-1)!\\
            &=(k-1)!(n-k-1)!\left(2n^2-3nk+2k^2-2\right)\\
            &=(k-1)!(n-k)!n+(k-1)!(n-k-1)!(2k^2-2nk+n^2-2)\\
            &\ge \left(\frac{n}2-1\right)!\left(\frac{n}2-1\right)!(n^2-2),
            \end{align*}
            since both summands in the second to last line are minimised for $k=n/2$.
            \item $F_1\subseteq F_2,F_3\subseteq F_4$, with $\big(|F_1|,|F_2|,|F_3|,|F_4|\big)=(k-1,k,k,k+1)$. This is the diamond shape, so we know from the proof of \Cref{diamond BLYM} that
            \[
            \Ssize(\F_i)=n(n-1)(k-1)!(k+1)!\ge \left(\frac{n}2-1\right)!\left(\frac{n}2-1\right)!(n^2-n).
            \]
            \item $F_1,F_2,F_3\subseteq F_4$, with $\big(|F_1|,|F_2|,|F_3|,|F_4|\big)=(k,k,k,k+1)$. Then
            \begin{align*}
            \Ssize(\F_i)
            &=3k!(n-k)!+(k+1)!(n-k-1)!-3k!(n-k-1)!\\
            &=k!(n-k-1)!(3n-2k-2)
            =2k!(n-k)!+k!(n-k-1)!(n-2)\\
            &\ge \left(\frac{n}2-1\right)!\left(\frac{n}2-1\right)!(n^2-n),
            \end{align*}
            since in the second to last line both summands are minimised for $k=n/2$.
            \item $F_1,F_2\subseteq F_3$ and $F_2\subseteq F_4$, with $\big(|F_1|,|F_2|,|F_3|,|F_4|\big)=(k,k,k+1,k+1)$. Then
            \begin{align*}
            \Ssize(\F_i)
            &=2k!(n-k)!+2(k+1)!(n-k-1)!-3k!(n-k-1)!
            =k!(n-k-1)!(2n-1)\\
            &\ge\left(\frac{n}2-1\right)!\left(\frac{n}2-1\right)!\left(n^2-\frac{n}2\right).
            \end{align*}
        \end{itemize}
        The third and fourth case both give the lowest bound. This seems to suggest that the optimal family could also be made up of fork-shaped components, that is, $F_1,F_2,F_3\subseteq F_4$, and not just diamond shapes. However, it seems likely that space constraints in the middle layer would make it impossible, and that in reality the diamond shape is the only optimal shape. However we cannot infer it from this proof.
    \end{enumerate}
    For $i=1,2,3,4$, write $c_i$ for the number of components of $\F$ of order $i$. We are maximising
    \begin{equation}
    \label{max}
    |\F|=c_1+2c_2+3c_3+4c_4.
    \end{equation}
    By our previous analysis, we can now replace in \Cref{sum} each $c_i$ with the lower bound we found for each component order:
    \begin{align*}
    &c_1\left(\frac{n}2\right)!\left(\frac{n}2\right)!+c_2n\left(\frac{n}2\right)!\left(\frac{n}2-1\right)!+\\
    &+c_3\left(\frac{n}2-1\right)!\left(\frac{n}2\right)!\left(\frac{3}2n-1\right)+c_4\left(\frac{n}2-1\right)!\left(\frac{n}2-1\right)!(n^2-n)\le n!.
    \end{align*}
    We can in turn rewrite this as
    \begin{align*}
    &c_1\left(\left(\frac{n}2\right)!\left(\frac{n}2\right)!\right)+2c_2\left(\left(\frac{n}2\right)!\left(\frac{n}2\right)!\right)+\\
    &+3c_3\left(\left(\frac{n}2\right)!\left(\frac{n}2-1\right)!\left(\frac{n}2-\frac{1}3\right)\right)+4c_4\left(\left(\frac{n}2\right)!\left(\frac{n}2-1\right)!\left(\frac{n}2-\frac{1}2\right)\right)\le n!
    \end{align*}
    The coefficient of $4c_4$ is the smallest, and so to maximise \Cref{max} we should set $c_1=c_2=c_3=0$ and $c_4$ as large as possible, which gives
    \[
    c_4\le\frac{n!}{\left(\frac{n}2-1\right)!\left(\frac{n}2-1\right)!(n^2-n)}=\binom{n-2}{n/2-1}.
    \]
    This then implies
    \[
    |\F|= 4c_4\le 4\binom{n-2}{n/2-1}.\qedhere
    \]
\end{proof}

We note that this approach breaks down for $k=2$ when $n$ is odd, because then a diamond shape centred on one of the middle layers is no longer the component for which $\Ssize(C)$ is minimised. If we could find a general lower bound on $\Ssize(C)$ this would at least give a weaker version of \Cref{main conjecture}. 

\section{Largest disconnected family}\label{sec:connected}

Rather than bounding component order, in this section we consider instead the largest possible family $\F\subseteq 2^{[n]}$ such that $G_\F$ is disconnected.
If you like, this can be seen as an extreme case of \Cref{main conjecture}, where $k$ is just shy of $n$.
In essence we are asking what is the largest $k$ (now not necessarily an integer) such that the bound on the size of the family is nontrivial, meaning that we can fit more than one component.
As \Cref{disconnected} shows, this $k$ gets arbitrarily close to $n$ as $n\to\infty$.

This proof of \Cref{disconnected} will be quite different from the ones in the previous sections, as it will be more direct. The main idea is to show that, given a family made up of two components, there is a minimum number of sets that cannot be in the family. To make it easier to read, we split the proof into a series of lemmas, which are structured as follows: first we define some necessary notation. \Cref{AB} and \Cref{F+F-} will give necessary preliminaries for \Cref{final lemma}, which encapsulates the main idea of the proof. \Cref{key lemma} is the engine of the proof, as this is where we do most of the work for computing a lower bound on the number of sets that have to be excluded. This, together with the technical \Cref{technical}, will prove \Cref{bound excluded}, which combined with \Cref{final lemma} completes the proof.

We will from now on assume that we are working with some fixed disconnected family $\F\subseteq 2^{[n]}$ that is optimal, i.e.~no sets can be added to $\F$ without making it connected.
    Given a family $\A$, we denote by $\partial^+(\A)$, the set of all supersets of some element of $\A$, that is,
    \[
    \partial^+(\A)=\{X\in 2^{[n]}:A\subseteq X\text{ for some }A\in\A\}.
    \]
    Similarly we define $\partial^-(\A)$ to be the set of all subsets of some element of $\A$, that is,
    \[
    \partial^-(\A)=\{X\in 2^{[n]}:X\subseteq A\text{ for some }A\in\A\}.
    \]
    For a set $A$, we also denote by $\partial(A)$ the set of all subsets of $A$ of size $|A|-1$.

    Suppose now $\F$ is split in two components $\A$ and $\B$, so that there are no comparable pairs $(A,B)$ with $A\in\A$ and $B\in\B$. Then we define the following sets:
    \begin{align*}
    \F^+&=\{F\in 2^{[n]}:F\notin (\partial^-(\A)\cup\partial^-(\B))\text{ and }\nexists F'\subsetneq F\text{ with }F'\notin (\partial^-(\A)\cup\partial^-(\B))\}, \text{ and}\\
    \F^-&=\{F\subseteq[n]:F\notin (\partial^+(\A)\cup\partial^+(\B))\text{ and }\nexists F'\supsetneq F\text{ with }F'\notin (\partial^+(\A)\cup\partial^+(\B))\}.
    \end{align*}

It will be useful to notice the following fact.

\begin{fact}
\label{AB}
    If $\F$ is split into $\A$ and $\B$ as above, then $\partial^+(\A)\cap\partial^-(\A)=\A$ and $\partial^+(\B)\cap\partial^-(\B)=\B$. Also $\partial^+(\A)\cap\partial^-(\B)=\emptyset$ and $\partial^-(\A)\cap\partial^+(\B)=\emptyset$. 
\end{fact}

\begin{proof}
    The first property is immediate from the definition. For the second property, if, for example, there were some set $X\in\partial^+(\A)\cap\partial^-(\B)$, then we would get the chain $A\subseteq X\subseteq B$ for some $A\in\A$ and $B\in\B$, which is a contradiction.
\end{proof}

The definition of $\F^+$ and $\F^-$ may seem quite arbitrary at face value. The idea is that they represent the point where the two components come very close to intersecting. This is illustrated in \Cref{F}. Defining $\F^+$ and $\F^-$ in this way is useful due to the following fact. 

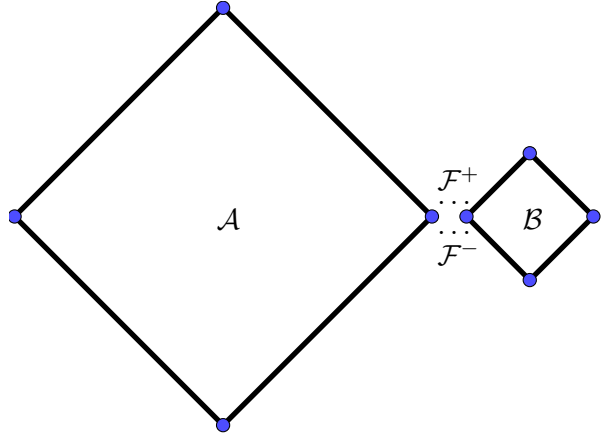
\begin{figure}
    \centering
\definecolor{ududff}{rgb}{0.30196078431372547,0.30196078431372547,1}
\definecolor{uuuuuu}{rgb}{0,0,0}
\begin{tikzpicture}[line cap=round,line join=round,>=triangle 45,x=0.45cm,y=0.45cm]
\clip(-17.05,-10.34) rectangle (2,3.6);
\draw [line width=2pt,color=uuuuuu] (-10.75,3.41)-- (-16.91,-2.76);
\draw [line width=2pt,color=uuuuuu] (-16.91,-2.76)-- (-10.75,-8.93);
\draw [line width=2pt,color=uuuuuu] (-10.75,-8.93)-- (-4.58,-2.76);
\draw [line width=2pt,color=uuuuuu] (-4.58,-2.76)-- (-10.75,3.41);
\draw [line width=2pt,color=uuuuuu] (-1.69,-0.89)-- (-3.56,-2.76);
\draw [line width=2pt,color=uuuuuu] (-3.56,-2.76)-- (-1.69,-4.64);
\draw [line width=2pt,color=uuuuuu] (-1.69,-4.64)-- (0.19,-2.76);
\draw [line width=2pt,color=uuuuuu] (0.19,-2.76)-- (-1.69,-0.89);
\draw (-4.72,-2) node[anchor=north west] {$\dots$};
\draw (-4.72,-2.9) node[anchor=north west] {$\dots$};
\draw (-11.24,-2.25) node[anchor=north west] {$\mathcal{A}$};
\draw (-2.2,-2.25) node[anchor=north west] {$\mathcal{B}$};
\draw (-4.72,-1) node[anchor=north west] {$\mathcal{F}^+$};
\draw (-4.72,-3.2) node[anchor=north west] {$\mathcal{F}^-$};
\begin{scriptsize}
\draw [fill=ududff] (-10.75,3.41) circle (2.5pt);
\draw [fill=ududff] (-16.91,-2.76) circle (2.5pt);
\draw [fill=ududff] (-10.75,-8.93) circle (2.5pt);
\draw [fill=ududff] (-4.58,-2.76) circle (2.5pt);
\draw [fill=ududff] (-1.69,-0.89) circle (2.5pt);
\draw [fill=ududff] (-3.56,-2.76) circle (2.5pt);
\draw [fill=ududff] (-1.69,-4.64) circle (2.5pt);
\draw [fill=ududff] (0.19,-2.76) circle (2.5pt);
\end{scriptsize}
\end{tikzpicture}
\vspace{-15pt}
    \caption{The sets $\F^+$ and $\F^-$.}
    \label{F}
\end{figure}

\begin{fact}
\label{F+F-}
    If $F\subseteq X$ for some $F\in\F^+$, then $X\notin\F$. If $X\subseteq F$ for some $F\in\F^-$, then $X\notin\F$.
\end{fact}
\begin{proof}
    Suppose $F\subseteq X$ for some $F\in\F^+$. Then $X\notin\partial^-(\A)$ and $X\notin\partial^-(\B)$. If this was not the case then also $F$ would be in either $\partial^-(\A)$ or $\partial^-(\B)$, which is a contradiction. But if $X$ were in $\F$, then it would be in either $\A$ or $\B$, and hence it would be in either $\partial^-(\A)$ or $\partial^-(\B)$. We conclude that $X\notin\F$. The proof of the second statement is analogous.
\end{proof}

This means that we can use $\F^+$ and $\F^-$ to lower bound the number of sets that cannot be in $\F$. More specifically we get the following lemma.

\begin{lemma}
\label{final lemma}
    If $\F$ consists of two components $\A$ and $\B$, then 
    \(
    |\F|\le 2^n-|\partial^+(\F^+)|-|\partial^-(\F^-)|.
    \)
\end{lemma}

\begin{proof}
    By \Cref{F+F-}, we know that none of the sets in $\partial^+(\F^+)$ or in $\partial^-(\F^-)$ can be in $\F$. The only thing left to observe is that $\partial^+(\F^+)\cap\partial^-(\F^-)=\emptyset$. This is by optimality of $\F$. Indeed, if there is some $F\in\partial^+(\F^+)\cap\partial^-(\F^-)$, then in particular $F\notin (\partial^-(\A)\cup\partial^-(\B))$ and $F\notin (\partial^+(\A)\cup\partial^+(\B))$. But then $F$ isn't comparable with any set in $\F$, hence we could add $F$ to $\F$, which is a contradiction since we assumed $\F$ to be optimal. Thus $\partial^+(\F^+)\cap\partial^-(\F^-)=\emptyset$, proving the lemma.
\end{proof}

We  now lower bound  $|\partial^+(\F^+)|+|\partial^-(\F^-)|$. First we observe $\F^+$ and $\F^-$ are nonempty.

\begin{claim}
    If $\F$ is disconnected, then $\F^+\ne\emptyset$ and $\F^-\ne\emptyset$.
\end{claim}

\begin{proof}
    We have $\emptyset\notin\F$ and $[n]\notin\F$, or else $\F$ would be connected. But now notice the following facts:
    \begin{enumerate}[(i)]
        \item $[n]\notin (\partial^-(\A)\cup\partial^-(\B))$, otherwise we would have $[n]\in\F$.
        \item $\emptyset\in(\partial^-(\A)\cup\partial^-(\B))$, since both components contain at least one set.
    \end{enumerate}
    Now if all proper subsets of $[n]$ are in either $\partial^-(\A)$ or $\partial^-(\B)$, then $[n]\in\F^+$ and we are done. Otherwise, we construct a chain of subsets $[n]\supseteq X_{n-1}\supseteq X_{n-2}\supseteq\dots\supseteq X_k$ such that each next subset has one fewer element than the previous and for each $k\le i\le n$ we have $X_i\notin(\partial^-(\A)\cup\partial^-(\B))$. Since $\emptyset\in(\partial^-(\A)\cup\partial^-(\B))$, the chain must stop at some point, i.e.~eventually we find a set $X_k$ such that $X_k\notin(\partial^-(\A)\cup\partial^-(\B))$ but all proper subsets of $X_k$ are in either $\partial^-(\A)$ or $\partial^-(\B)$. Then $X_k\in\F^+$.
    
    We can use the same reasoning for $\F^-$ starting at the empty set, and we will eventually have to find some set in $\F^-$.
\end{proof}

Now we can use the following lemma to show both that $\F^+$ and $\F^-$ have a certain minimum size, and that they contain sets that are very close to each other in size.

\begin{lemma}
\label{key lemma}
    If $F\in\F^+$, then $\F^-$ must contain at least $|F|-1$ sets, each of size at least $|F|-2$. Similarly, if $F\in\F^-$, then $\F^+$ must contain at least $n-|F|-1$ sets, each of size at most $|F|+2$.
\end{lemma}

In order to prove the lemma, we first prove another claim.

\begin{claim}
\label{one in A one in B}
    If $F\in\F^+$, then $\partial(F)$ contains at least one set in $\partial^-(\A)\backslash\partial^-(\B)$ and at least one set in $\partial^-(\B)\backslash\partial^-(\A)$.
\end{claim}

\begin{proof}
    Assume $\partial(F)\subseteq\partial^-(\A)$. We know that $F\notin\partial^-(\A)$ and $F\notin\partial^-(\B)$. Then we must have $F\in\partial^+(\B)$, otherwise we could safely add $F$ to $\A$, which breaks optimality of $\F$. But this is a contradiction: if $F\in\partial^+(\B)$, then in particular $F'\in\partial^+(\B)$ for some $F'\subseteq F$ with $|F'|=|F|-1$. But we assumed that $F'\in\partial^-(\A)$, so we get a contradiction by \Cref{AB}. So we cannot have $\partial(F)\subseteq\partial^-(\A)$, and similarly we cannot have $\partial(F)\subseteq\partial^-(\B)$. But since each set in $\partial(F)$ is in either $\partial^-(\A)$ or $\partial^-(\B)$, we deduce that $\partial(F)$ must contain at least one set in $\partial^-(\A)\backslash\partial^-(\B)$ and at least one set in $\partial^-(\B)\backslash\partial^-(\A)$.
\end{proof}

\begin{proof}[Proof of \Cref{key lemma}]
    We only show the first statement, as the second is analogous. So let $F_1\in\partial F\cap(\partial^-(\A)\backslash\partial^-(\B))$ and let $F_2\in\partial F\cap\partial^-(\B)$. We can do this by \Cref{one in A one in B}. Now consider the set $F_1\cap F_2$. First, notice that $|F_1\cap F_2|=|F|-2$. We claim that $F_1\cap F_2\notin\partial^+(\A)$ and $F_1\cap F_2\notin\partial^+(\B)$. If it were in $\partial^+(\A)$, then $F_2$ would also be in $\partial^+(\A)$, which is a contradiction by \Cref{AB} since $F_2\in\partial^-(\B)$. Similarly if $F_1\cap F_2$ were in $\partial^+(\B)$, then $F_1$ would also be in $\partial^+(\B)$, which is a contradiction since $F_1\in\partial^-(\A)$. So we deduce that $F_1\cap F_2\notin (\partial^+(\A)\cup\partial^+(\B))$, which implies that either it is in $\F^-$, or it has a superset in $\F^-$.

    So we have found a set in $\F^-$ of size at least $|F|-2$. To show that there must actually be at least $|F|-1$ such sets, we observe that we could actually have picked multiple pairs $(F_1,F_2)$ in the beginning of the proof and by the same reasoning obtain more sets in $\F^-$. However, it could be that the same set in $\F^-$ corresponds to multiple pairs $(F_1,F_2)$. It turns out that we can still find at least $|F|-1$ pairs that each lead to distinct sets in $\F^-$. To do this, we pick a set of $\ell$ distinct pairs $\{(F_1,F_2)^{(1)},\dots,(F_1,F_2)^{(\ell)}\}$, with $F_1^{(i)}\in\partial F\cap\partial^-(\A)$ and $F_2^{(i)}\in\partial F\cap\partial^-(\B)$ for all $i$, such that for any two pairs $(F_1,F_2)^{(i)}$ and $(F_1,F_2)^{(j)}$ one of the following holds:
    \begin{enumerate}[(1)]
        \item\label{itm:distinct} $F_1^{(i)},F_1^{(j)},F_2^{(i)},F_2^{(j)}$ are all distinct,
        \item\label{itm:F1i} $F_1^{(i)}=F_1^{(j)}$ and $F_1\in\partial F\cap(\partial^-(\A)\backslash\partial^-(\B))$, or
        \item\label{itm:F2i} $F_2^{(i)}=F_2^{(j)}$ and $F_2\in\partial F\cap(\partial^-(\B)\backslash\partial^-(\A))$.
    \end{enumerate}

    If we can find $\ell$ pairs satisfying these conditions, then each pair will lead to a distinct set $F^{(i)}\in\F^-$. We show this for each of the three cases:

    \begin{enumerate}
        \item[\eqref{itm:distinct}] From $F_1^{(i)}\cap F_2^{(i)}$ and $F_1^{(j)}\cap F_2^{(j)}$ we derive sets $F^{(i)},F^{(j)}\in\F^-$ as described previously. Now assume $F^{(i)}=F^{(j)}$. Since the four sets $F_1^{(i)},F_2^{(i)},F_1^{(j)},F_2^{(j)}$ are all distinct and they are all subsets of $F$ of size $|F|-1$, we deduce that each is missing a distinct element of $F$, thus $F_1^{(i)}\cap F_2^{(i)}$ is missing two elements of $F$, which must be contained in $F_1^{(j)}\cap F_2^{(j)}$. But then $(F_1^{(i)}\cap F_2^{(i)})\cup(F_1^{(j)}\cap F_2^{(j)})=F$. So since $F^{(i)}\supseteq F_1^{(i)}\cap F_2^{(i)}$ and $F^{(j)}\supseteq F_1^{(j)}\cap F_2^{(j)}$, we must have $F^{(i)}=F^{(j)}\supseteq F$. But this is a contradiction, as then (by definition of $\F^-$) we would have $F\notin(\partial^+(\A)\cup\partial^+(\B))$, but we also know that $F\notin(\partial^-(\A)\cup\partial^-(\B))$ (since $F\in\F^+$). This contradicts the optimality of $\F$, as we could safely add $F$ to $\F$. It follows that $F^{(i)}\ne F^{(j)}$.
        \item[\eqref{itm:F1i}] Like before, assume $F^{(i)}=F^{(j)}$. This time we have $(F_1^{(i)}\cap F_2^{(i)})\cup(F_1^{(j)}\cap F_2^{(j)})=F_1^{(i)}=F_1^{(j)}$. Therefore $F^{(i)}=F^{(j)}\supseteq F_1^{(i)}=F_1^{(j)}$. But this is also a contradiction, since then we would have $F_1^{(i)}\notin \partial^+(\B)\cup\partial^+(\A)$ (by definition of $\F^-$), but we also have $F_1^{(i)}\notin\partial^-(\B)$ by assumption. This contradicts the optimality of $\F$, as we could safely add $F_1^{(i)}$ to $\A$.
        \item[\eqref{itm:F2i}] This is similar to~\eqref{itm:F1i}, with the roles of $\A$ and $\B$ inverted.
    \end{enumerate}

    Therefore it only remains to show that there exists a set of pairs of size $\ell\ge |F|-1$ satisfying conditions \eqref{itm:distinct}--\eqref{itm:F2i}. We can achieve this by choosing
    \begin{itemize}
        \item all pairs $(F_1,F_2)$ with $F_1\in\partial^-(\A)\backslash\partial^-(\B)$ and $F_2\in\partial^-(\B)\backslash\partial^-(\A)$, and
        \item a pair $(F_1,F_2)$ for each $F_2\in\partial^-(\A)\cap\partial^-(\B)$, with $F_1$ being any set in $\partial^-(\A)\backslash\partial^-(\B)$. 
    \end{itemize}
    
    In total we get
    \begin{align*}
    \ell
    &=|\partial^-(\A)\backslash\partial^-(\B)|\cdot|\partial^-(\B)\backslash\partial^-(\A)|+|\partial^-(\A)\cap\partial^-(\B)|\\
    &\ge|F|-|\partial^-(\A)\cap\partial^-(\B)|-1+|\partial^-(\A)\cap\partial^-(\B)|=|F|-1.
    \end{align*}
    The inequality holds because each of $\partial^-(\A)\backslash\partial^-(\B)$ and $\partial^-(\B)\backslash\partial^-(\A)$ has at least one element, and $|\partial^-(\A)\backslash\partial^-(\B)|+|\partial^-(\B)\backslash\partial^-(\A)|=|F|-|\partial^-(\A)\cap\partial^-(\B)|$, the worst case being therefore when $|\partial^-(\A)\backslash\partial^-(\B)|=1$ and $|\partial^-(\B)\backslash\partial^-(\A)|=|F|-|\partial^-(\A)\cap\partial^-(\B)|-1$ (or vice versa).

    Since each of the $\ell$ pairs leads to a distinct set in $\F^-$ with size at least $|F|-2$, this concludes the proof of the first part of the lemma. The argument for the second statement is exactly the same where we invert the roles of subset and superset.
\end{proof}

This lemma is sufficient to lower bound the size of $|\partial^+(\F^+)|+|\partial^-(\F^-)|$. For this, we use a technical lemma which is essentially a corollary of the Kruskal--Katona theorem (\Cref{thm:KrKa}).

\begin{lemma}
\label{technical}
If $S$ is a set of $k+1$ sets of size at least $k$, then $\partial^-(S)\ge 2^{k+1}-1$.

If $S'$ is a set of $k$ sets of size at least $k$, then $\partial^-(S')\ge 2^{k+1}-2$.

We also have the counterparts for supersets, that is, if $S$ is a set of $n-k+1$ sets of size at most $k$, then $\partial^+(S)\ge 2^{n-k+1}-1$, and, if $S'$ is a set of $n-k$ sets of size at most $k$, then $\partial^+(S')\ge 2^{n-k+1}-2$. 
\end{lemma}

\begin{proof}
    The proof works by adding up the bound given by Kruskal--Katona layer by layer. Clearly if the sets in $S$ and $S'$ are larger than $k$, then the number of subsets is even larger, so we may assume they all have size exactly $k$. Conveniently, $|S|=\binom{k+1}k$, therefore by Kruskal--Katona there are at least $\binom{k+1}i$ subsets of elements of $S$ on layer $i$, for $0\le i\le k$. If we add these up, we get in total
    \[
    \partial^-(S)\ge \sum_{i=0}^k\binom{k+1}i=2^{k+1}-1.
    \]

    For $S'$, we have $|S'|=\binom{k}k+\binom{k-1}{k-1}+\dots+\binom{1}1$. Hence there are at least
    \[
    \sum_{j=0}^{i-1}\binom{k-j}{k-i-j}
    \]
    subsets of elements of $S'$ on layer $i$. If we add these up, we get in total
    \[
    \partial^-(S')\ge\sum_{i=0}^k\sum_{j=0}^{i-1}\binom{k-j}{k-i-j}=\sum_{j=0}^{k-1}\sum_{i=0}^{k-j}\binom{k-j}i=\sum_{j=0}^{k-1}2^{k-j}\ge
    2^{k+1}-2.
    \]
    The proof for the converse is analogous, and it uses an analogous version of Kruskal--Katona for lower bounding the number of supersets instead of subsets.
\end{proof}

The technical lemma will be useful in the proof of this final lemma.

\begin{lemma}
\label{bound excluded}
\[
|\partial^+(\F^+)|+|\partial^-(\F^-)| \ge
\begin{cases}
2^{n/2+1}-2 & \text{ if $n$ is even, and}\\
3\cdot 2^{(n-1)/2}-2 &  \text{ if $n$ is odd.}
\end{cases}
\]
\end{lemma}

\begin{proof}
    Let $F$ be the smallest set in $\F^+$, with $|F|=k$. By \Cref{key lemma}, there are at least $k-1$ sets in $\F^-$, each of size at least $k-2$. Suppose any of these sets had size at least $k$. Then we would have
    \[
    |\partial^+(\F^+)|+|\partial^-(\F^-)|\ge 2^{n-k}+2^{k}.
    \]
    This is minimised by $2^{n/2+1}>2^{n/2+1}-2$ when $n$ is even and by $3\cdot 2^{(n-1)/2}>3\cdot 2^{(n-1)/2}-2$ when $n$ is odd.

    So from now on we may assume that there are at least $k-1$ sets in $\F^-$ of sizes $k-2$ and $k-1$. Suppose there is one of size $k-2$. Then by \Cref{key lemma}, there are at least $n-k+1$ sets in $\F^+$ of size at most $k$. Then by \Cref{technical}, we must have
    \[
    |\partial^+(\F^+)|+|\partial^-(\F^-)|\ge 2^{n-k+1}-1+2^{k-1}-1.
    \]
    This is minimised with $k=n/2$ if $n$ is even and $k=(n+1)/2$ if $n$ is odd, giving exactly the bound of this lemma.

    The only case we are left with is when all the at least $k-1$ sets in $\F^-$ we found have size $k-1$. Then, again by \Cref{technical}, we have
    \[
    |\partial^+(\F^+)|+|\partial^-(\F^-)|\ge 2^{n-k}+2^k-2,
    \]
    which is minimised with $k=n/2$ if $n$ is even and $k=(n-1)/2$ if $n$ is odd, giving once more the exact bound of this lemma. This concludes the proof.
\end{proof}

Now we finally have all the ingredients to prove \Cref{disconnected}.

\begin{proof}[Proof of \Cref{disconnected}]
    By \Cref{final lemma,bound excluded}, we deduce that
    \[
    |\F| \le 2^n-|\partial^+(\F^+)|-|\partial^-(\F^-)|\le
    \begin{cases}
2^n-2^{n/2+1}+2 & \text{ if $n$ is even, and}\\
2^n-3\cdot 2^{(n-1)/2}+2 &  \text{ if $n$ is odd.}
\end{cases}
    \]
Sharpness is witnessed by the family $\F=\A\cup\B$ with $\A=\{[n/2]\}$ if $n$ is even and $\A=\{[(n-1)/2]\}$ if $n$ is odd, and $\B=\{B\in 2^{[n]}:B\nsubseteq A\text{ and }A\nsubseteq B \text{ for any }A\in\A\}$.
\end{proof}

\section{Concluding remarks}

Recall the narrative we began with in the introduction.
As we cross the cardinality threshold $\La(n)=\binom{n}{\lfloor n/2\rfloor}$, how does the space of $2$-chains in a subfamily of the Boolean lattice evolve?
In this work, we interpreted this rough question in terms of largest component in the comparability graphs $G_\F$ of subfamilies $\F \subseteq 2^{[n]}$,
 as we tune the prescribed lower bound on $|\F|$ beyond $\La(n)$.
Because of the middle layer family $\binom{[n]}{\lfloor n/2\rfloor}$, $G_\F$ may consist of a disjoint union of isolated vertices unless we prescribe that $|\F| > \La(n)$.
Our main result (\Cref{main}) implies the following statement.
Given $\varepsilon>0$, if $\F \subset 2^{[n]}$ is a family satisfying $|\F| \ge (1+\varepsilon)\La(n)$, then $G_\F$ must have a component of order $\exp(\Omega(n/\log n))$.
Qualitatively, this already shows how $\La(n)$ constitutes a sharp threshold with respect to the component structure of the comparability graph, in that up to cardinality $\La(n)$ it can be all isolated vertices, while a modest $1+\varepsilon$ factor above $\La(n)$ cardinality guarantees a strikingly large component.

Our main conjecture (\Cref{main conjecture}) proposes that the size guarantee $\exp(\Omega(n/\log n))$ above could be improved to $\exp(\Omega(n))$.
(Consider the Stirling's approximation computations displayed parenthetically after the conjecture.)
This improvement would in fact follow if we could establish that $\mad^\star(\comporder,\mathcal{O}^\star)= O(\log \comporder)$, which is an open problem mentioned in~\cite{ABSZZ+}, or, more specifically, if we could show $\xi^\star_n(\comporder)=O(\log \comporder)$, which is related to \Cref{question 5.1} and is an interesting problem in its own right.
We contend that the special case $k=n-2$ of \Cref{main conjecture} alone is enticing, even only asymptotically.
Moreover, \Cref{main conjecture} suggests that the size of a largest component will be of order $2^{n-o(n)}$ once we prescribe that $|\F| = \omega(\La(n))$.
Thus, the broader picture sketched (if correct) compares interestingly against the intuition we might have borrowed from, say, the theory of random graphs or percolation.

As a side remark, we note how the study of such `condensation' phenomena has been pursued quite intensively in stochastic, algorithmic contexts, see e.g.~\cite{AcCo08}, but to our knowledge ---and to our surprise--- not systematically  in the context of (extremal) combinatorics. In a parallel work~\cite{BHK+}, Buys, van den Heuvel and the second author establish a similar (but qualitatively distinct) phenomenon with respect to Tur\'an's theorem.

The study undertaken here not only links nicely to many classic studies, but also raises numerous interesting research questions. The first author's thesis contains an account of several of these. 
To conclude the paper, we highlight just one, which takes an alternative, and perhaps more authentic, interpretation of our original motivating question.

Given a positive integer $\compsize$, how large can $\F$ be before $G_\F$  must contain some component of {\em size} greater than $\compsize$?
That is, instead of measuring the number of elements of $\F$ it spans, we measure the number of edges (i.e.~comparable pairs) a component of $G_\F$ has.
While this alternative interpretation is related to the problems we tackled in this paper, it does not seem that we can straightforwardly apply the same techniques.
Nevertheless, we ask, could the disjoint union of complete subfamilies be extremal for this question?

\subsection*{Open access statement} For the purpose of open access,
a CC BY public copyright license is applied
to any Author Accepted Manuscript (AAM)
arising from this submission.

\bibliographystyle{abbrv}
\bibliography{references}

\end{document}